\documentclass[reqno,12pt,letterpaper]{amsart}
\usepackage{amsmath,amssymb,amsthm,graphicx,mathrsfs,url}
\usepackage{enumitem}

\usepackage{graphicx}
\usepackage{amsmath, amscd}
\usepackage[usenames,dvipsnames]{color}
\usepackage[colorlinks=true, linkcolor=Red, citecolor=Green]{hyperref}

\newtheorem{theo}{Theorem}[section]
\newtheorem{prop}{Proposition}[section]

\newtheorem{lemm}[prop]{Lemma}
\theoremstyle{definition}

\newtheorem{rem}[prop]{Remark}

\numberwithin{equation}{section}

\newcommand{\R}{\mathbb{R}}
\newcommand{\N}{\mathbb{N}}
\newcommand{\Z}{\mathbb{Z}}
\newcommand{\C}{\mathbb{C}}
\newcommand{\dd}{\mathrm{d}}
\newcommand{\e}{\mathrm{e}}
\newcommand{\wt}{\widetilde}
\newcommand{\ubf}{\mathbf{u}}
\newcommand{\Pbf}{\mathbf{P}}
\newcommand{\Rbf}{\mathbf{R}}

\let\Re=\Real\def\clos{{\rm clos}\:}
\def\Res{{\rm Res}\:}

\DeclareMathOperator{\tr}{tr}
\DeclareMathOperator{\id}{Id}

\let\Im=\Imag
\def\C{\mathbb {C}}
\def\ep{\epsilon}

\def\tg{T_{\gamma}}
\def\ag{a_{\gamma}}
\def\P{{\rm P}}

\title[On the number of poles]{On the number of poles of the dynamical zeta functions for billiard flow}

\author[V.Petkov]{Vesselin Petkov}
\address{Université de Bordeaux, Institut de Mathématiques de Bordeaux, 351, Cours de la Libération, 33405 Talence, France}
\email{petkov@math.u-bordeaux.fr}

\begin{document}

\maketitle
\begin{abstract} We study the number of the poles of the meromorphic continuation of the dynamical zeta functions $\eta_N$ and $\eta_D$ for several strictly convex disjoint obstacles satisfying non-eclipse condition. For $\eta_N$ we obtain a strip $\{z \in \C:\: \Re s  > \beta\}$ with infinitely many poles. For $\eta_D$ we prove the same result assuming the boundary real analytic. Moreover, for $\eta_N$ we obtain a characterisation of $\beta$ by the pressure $\P(2G)$ of some function $G$ on the space $\Sigma_A^f$  related to the dynamical characteristics of the obstacle.

\end{abstract}
\section{Introduction}
Let $D_1, \dots, D_r \subset \R^d,\: {r \geqslant 3},\: d \geqslant 2,$ be compact strictly convex disjoint obstacles with $C^{\infty}$ smooth boundary and let $D = \bigcup_{j= 1}^r D_j.$  We assume that every $D_j$ has non-empty interior and throughout this paper we suppose the following non-eclipse condition
\begin{equation}\label{eq:1.1}
D_k \cap {\rm convex}\: {\rm hull} \: ( D_i \cup D_j) = \emptyset, 
\end{equation} 
for any $1 \leqslant i, j, k \leqslant r$ such that $i \neq k$ and $j \neq k$.
Under this condition all periodic trajectories for the billiard flow in $\Omega  = \R^d \setminus \mathring{D}$ are ordinary reflecting ones without tangential intersections to the boundary of $D$.  We consider the (non-grazing) billiard flow {$\varphi_t$} (see Section 2 for the definition).  Next the periodic trajectories will be called periodic rays. For any periodic ray $\gamma$, denote by  $\tau(\gamma) > 0$ its period, by $\tau^\sharp(\gamma) > 0$ its primitive period, and by $m(\gamma)$ the number of reflections of $\gamma$ at the obstacles. Denote by  $P_\gamma$ the associated linearized Poincar\'e map (see section 2.3 in \cite{petkov2017geometry} and Section 2 for the definition). Let $\mathcal{P}$ be the set of all oriented periodic rays. Notice that some periodic rays have only one orientation, while others admits two (see \cite[\S2.3]{chaubet2022} for more details). Let $\Pi$ be the set of all primitive periodic rays. Then the counting function of the lengths of periodic rays satisfies 
\begin{equation} \label{eq:1.2}
\sharp\{\gamma \in \Pi:\: \tau^\sharp(\gamma) \leqslant x\} \sim \frac{\e^{h x}}{h x} , \quad x \to + \infty,
\end{equation}
for some $h > 0$ (see for instance, \cite[Theorem 6.5] {Parry1990} for weakly mixing suspension symbolic flow and \cite{ikawa1990poles}, \cite{morita1991symbolic}). Hence there exists an infinite number of primitive periodic trajectories and applying (\ref{eq:1.2}), for every sufficiently small $\ep > 0$ one obtains  the estimate
\[e^{(h - \ep)x} \leq \sharp\{\gamma \in \mathcal P:\: \tau(\gamma) \leqslant x\} \leqslant \e^{(h + \ep) x}, \: x \geq C_{\ep} \gg 1.\]

Moreover, for some positive constants $c_1, C_1, f_1, f_2$ we have (see for instance \cite[Appendix]{petkov1999zeta} and (\ref{eq:A.1}))
\[c_1\e^{f_1 \tau(\gamma)} \leqslant |\det(\mathrm{Id} - P_{\gamma})| \leqslant C_1\e^{f_2 \tau(\gamma)}, \quad \gamma \in \mathcal{P}.
\]
By using these estimates, define for $\Re(s) \gg 1$ two Dirichlet series
\[
\eta_\mathrm{N}(s) = \sum_{\gamma \in \mathcal{P}} \frac{\tau^\sharp(\gamma)\e^{-s\tau(\gamma)}}{|\det(\mathrm{Id}-P_\gamma)|^{1/2} }, \quad \eta_\mathrm{D}(s) = \sum_{\gamma \in \mathcal{P}} (-1)^{m(\gamma)} \frac{ \tau^\sharp(\gamma) \e^{-s\tau(\gamma)}}{|\det(\mathrm{Id}-P_\gamma)|^{1/2} },
\]
where the sums run over all oriented periodic rays. The length $\tau^{\sharp}(\gamma)$, the period $\tau(\gamma)$ and $|\det(\mathrm{Id} - P_{\gamma})|^{1/2}$ are independent of the orientation of $\gamma.$ We consider also for $q \geq 1, \: q \in \N,$ the zeta function
\[\eta_\mathrm{q}(s) = q\sum_{\gamma \in \mathcal{P}, m(\gamma) \in q \N} \frac{\tau^\sharp(\gamma)\e^{-s\tau(\gamma)}}{|\det(\mathrm{Id}-P_\gamma)|^{1/2} },\Re s \gg 1.\]
Clearly, $\eta_N(s) = \eta_1(s).$ These zeta functions are important for the analysis of the distribution of the scattering resonances related to the Laplacian in $\R^d \setminus \bar{D}$ with Dirichlet and Neumann boundary conditions on $\partial D$ (see \cite[\S1]{chaubet2022} for more details).

It was proved in \cite[Theorem 1.1 and Theorem 4.1]{chaubet2022} that $\eta_q$  admit a meromorphic continuation to $\C$ with simple poles and integer residues. We have the equality
\begin{equation} \label{eq:1.3}
\eta_D(s) = \eta_2(s) - \eta _1(s),\: \Re s \gg 1,
\end{equation}
hence $\eta_D$ admits also a meromorphic continuation to $\C$ with simple poles and integer residues.
 The functions $\eta_q(s)$ are Dirichlet series with positive coefficients and by a classical theorem of Landau (see for instance, \cite[Théorème 1, Chapitre IV]{bernstein1933}) they have a pole $s = a_q$, where $a_q$ is the abscissa of convergence of $\eta_q(s)$. On the other hand, from (\ref{eq:1.3}) it follows that some cancelations of  poles are possible. In this direction, for $d = 2$  \cite{stoyanov2001spectrum} and for $d\geqslant 3$ under some conditions  \cite{stoyanov2012non} Stoyanov proved that there exists $\varepsilon > 0$ such that $\eta_{\mathrm D}(s)$ is analytic for $\Re s \geqslant a_1 - \varepsilon.$ The same result has been proved for $d = 3$ and $a_1 > 0$ by Ikawa \cite{ikawa2000resonances}.
      
      The purpose of this paper is to prove that $\eta_q(s)$ has an infinite number of poles and to estimate $\beta \in \R$  such that the number of poles with $\Re s > \beta$ is infinite. The same questions are more difficult for $\eta_D(s)$ since the existence of at least one pole has been established only for obstacles with real analytic boundary  \cite[Theorem 1.3]{chaubet2022} and for obstacles with sufficiently small diameters \cite{ikawa1990zeta}, \cite{stoyanov2009poles}. Clearly, $a_q \leq a_1$. We have $a_2 = a_1$, since if $a_2 < a_1$, the function $\eta_D$ will have a singularity at $a_1$ which is impossible because $\eta_D$ is analytic for $\Re s \geq a_1$ (see \cite[Theorem 1]{petkov1999zeta}).

 Denote by $\Res \eta_q$, $\Res \eta_D$ the set of poles of $\eta_q$ and $\eta_D$, respectively.       
 We prove the following theorem.
      
      \begin{theo} For every $0 < \delta < 1$ there exists $\alpha_{\delta, q} < a_q$ such that for $\alpha < \alpha_{\delta, q}$ we have
\begin{equation} \label{eq:1.4}
\sharp \{ \mu_j \in \Res \eta_q: \: \Re \mu_j \geq  {\alpha}, \: |\mu_j | \leqslant r \} \neq  {\mathcal O}(r^{\delta}).
\end{equation} 
If $\eta_D$ cannot be prolonged as entire function, the same result holds for $\Res \eta_D.$
      \end{theo}
More precisely,  we show that for any $0 < \delta < 1$ there exists $\alpha_{\delta, q} < 0$ depending on the dynamical characteristics of $D$ such that if $\alpha < \alpha_{\delta, q}$, for 
 any constant $0 < C < \infty$ the estimate
\[\sharp \{ \mu_j \in \Res \eta_q: \: \Re \mu_j  \geq \alpha, \: |\mu_j | \leqslant r \} \leqslant C r^{\delta},\quad r \geqslant 1\]
does not hold. Similar results have been proved for Pollicott-Ruelle resonances for Anosov flows \cite[Theorem 2]{zworski2017local} and  for Axiom A flows \cite[Theorem 4.1] {jin2023resonances}. For obstacles $D$ satisfying (\ref{eq:1.1}) the same result for scattering resonances has been proved for Neumann problem in  \cite{petkov2002poles} and for Dirichlet problem and real analytic boundary in \cite[Theorem 1.3]{chaubet2022}. According to Theorem 1.1, it follows that for large $A > 0$ in the region ${\mathcal D}_A = \{z \in \C: \Re z > -A\}$ there are infinitely many poles $\mu \in \Res \eta_1 \cap {\mathcal D}_A$ and infinitely many poles $\nu \in \Res \eta_2 \cap {\mathcal D}_A.$ Therefore if $\eta_D$ is analytic in ${\mathcal D}_A$, by (\ref{eq:1.3}) we deduce that we must have an {\it infinite number of cancellations} of poles $\mu$ with poles $\nu$ and the corresponding residues of the cancelled poles $\mu$ and $\nu$ must coincide.  
    \begin{rem} The proof of Theorem 1.1 for $\eta_D$ works under the condition that there exist two sequences  $(\ell_j), (m_j)$ with $\ell_j \to \infty,\: m_j \to \infty$ as $j \to \infty$  with the properties in Proposition 4.1. If  $\eta_D$ cannot be prolonged as entire function, the  existence of such sequences  has been established by Ikawa (see \cite[Prop. 2.3]{ikawa1990poles}). In Theorem 1.1 we prove the inverse result. \end{rem}
     It is interesting to find the supremum of numbers $\beta_q < a_q$ such that the strip $\{z \in \C :\: \Re z > \beta_q\}$ contains an
     infinite number poles of $\eta_q$ and to obtain so called {\it essential spectral gap}. This is a difficult open problem. Let $b_q< a_q$ be the abscissa of convergence of the series
     \begin{equation} \label{eq:1.5}
     \sum_{\gamma, m(\gamma) \in q\N} \frac{\tau^{\sharp}(\gamma) e^{-s \tau(\gamma)}}{|\det(\mathrm{Id}-P_\gamma)|}, \: \Re s \gg 1
     \end{equation}
     and let $\alpha = \max\{0, a_1\}.$
      In our second result we obtain a more precise result for $\Res \eta_1$.
      \begin{theo} For any small $\ep > 0$ we have
\begin{equation} \label{eq:1.6}
\sharp\{\mu_j \in \Res \eta_1:\:\Re \mu_j > (2d^2 +2 d -1/2)(b_1 -2\alpha)- \ep\} = \infty.
\end{equation}
\end{theo}     
 Notice that \[2d^2 + 2d - 1/2 = 2(d^2 + d - 1) + 3/2 = 2 \dim G + 3/2,\] where $G$ is the $(d-1)$-Grassmannian bundle introduced in Section 2.  In Appendix, we prove that $b_1$ coincides with the abscissa of convergence of the series
   \begin{equation} \label{eq:1.7}
     \sum_{\gamma} \frac{\tau^{\sharp}(\gamma) e^{-s \tau(\gamma)}}{|\det(D_x \varphi_{\tau(\gamma)}\vert_{E_u(x)})|}, \: \Re s \gg 1,
     \end{equation} 
     where $E_u(x)$ is the unstable space of $x \in \gamma$ (see (\ref{eq:2.1}) for the notation).
  By using symbolic dynamics, we define (see (\ref{eq:A.3})) a function $G(\xi, y)< 0$ on the space $\Sigma_A^f$ related to dynamical characteristics of $D$ (see Appendix for definitions) and  prove the following
  \begin{prop} The abscissas of convergence $a_1$ and $b_1$ are given by
  \begin{equation}
  a_1 = \P(G),\: b_1 = \P(2 G),
  \end{equation}
  $\P(G)$ being the pressure of $G$ defined by $(\ref{eq:A.2})$.
  \end{prop}
     For $a_1 \leq 0,$ we have $\alpha = 0,$ and Theorem 1.2 is similar to \cite[Theorem 3]{zworski2017local} established for weakly mixing Anosov flows $\psi_t$, where instead of $b_1 = \P(2 G)$ one has the pressure $\P(2 \psi^u) < 0$ of the Sinai-Ruelle-Bowen potential
  \[\psi^u(x) = -\frac{d}{dt} \Bigl( \log |\det D_x\psi_t \vert_{E_u(x)}|\Bigr)\vert_{t = 0}.\]  
  Notice that for Anosov flow one has $\P(\psi^u) = 0$, (see \cite[Theorem 5]{bowen1975ergodic}), while in our case $a_1 = \P(G)$ can be different from 0.  More precise results for the poles of the semi-classical zeta function  for contact Anosov flows have been obtained in \cite{faure2013band}, \cite[Theorem 1.2]{faure2017semiclassical}.
  
  \begin{rem}The constant $2d^2 + 2d -1/2$ in (\ref{eq:1.6}) is related to the estimate (\ref{eq:3.3}) of Fourier transform $\hat{F}_{A, 1}$ in the local trace formula for $\eta_1(s)$ in Theorem 3.2 and is probably  not optimal. A better estimate of $\hat{F}_{A,1}$ can be obtained if the bound of the number of poles (\ref{eq:3.1}) is improved (see for example, \cite{arnoldi2017}, where the Hausdorff dimension of the trapped set $K$ is involved).
  \end{rem}
 We have $b_1 = b_2$ since the series 
 \begin{equation} \label{eq:1.9}
 \sum_{\gamma} \frac{(-1)^{m(\gamma)} \tau^{\sharp}(\gamma)e^{- s \tau(\gamma)}}{|\det(\id - P_{\gamma})|},\: \Re s \gg 1
 \end{equation}
  is analytic for $\Re s \geq b_1$. We discuss this question at the end of the Appendix.  Theorem 1.2 can be generalized for $\Res \eta_2$ and one obtains (\ref{eq:1.6}). The proof works with some modifications.

 The paper is organised as follows. In Section 2 we collect some definitions and notations from \cite{chaubet2022} which are necessary for the exposition. In particular, we define the non-grazing billiard flow $\varphi_t$, the $(d-1)$-Grassmannian bundle $G$, the bundles $\mathcal E_{k, \ell}$ over $G,$ and the operators $\Pbf_{k, \ell},\: 0 \leq k \leq d,\: 0 \leq \ell \leq d^2.$ In  Section 3 we obtain local trace formulas combining the results in {\cite[\S6.1]{jin2023resonances} and \cite[Lemma 3.1]{chaubet2022}. In Section 4  we prove Theorems 1.1 and 1.2. Finally in the Appendix we use symbolic dynamics and establish Proposition 1.2.

 \section{Preliminaries}
We recall the definition of billiard flow $\phi_t$ described in \cite[\S 2.1]{chaubet2022}.  Denote by $S\R^d$ the unit tangent bundle of $\R^d$ and by $\pi : S\R^d \to \R^d$ the natural projection. For $x \in \partial D_j$, denote by $n_j(x)$ the {\it inward unit normal vector} to $\partial D_j$ at the point $x$ pointing into $ D_j.$ Set 
\[\mathcal D= \{(x,v) \in S \R^d~:~x \in \partial D\}.\]
We say that $(x,v) \in T_{\partial D_j}(\R^d)$ is incoming (resp. outgoing) if we have $\langle v, n_j(x)\rangle > 0$ (resp. $\langle v, n_j(x) \rangle < 0$). Introduce
\[
\begin{aligned}
\mathcal D_\mathrm{in} &= \{ (x, v) \in \mathcal D~:~(x,v) \text{ is } {\rm incoming}\}, \\
\mathcal D_\mathrm{out}&= \{ (x, v) \in \mathcal D~:~ (x,v) \text{ is } {\rm outgoing}\}.
\end{aligned}
\]
Define the grazing set $\mathcal D_\mathrm{g}= T(\partial D) \cap \mathcal D$ and obtain
\[\mathcal D = \mathcal D_\mathrm{g} \sqcup \mathcal D_\mathrm{in} \sqcup \mathcal D_\mathrm{out}.\]
 The billiard flow $(\phi_t)_{t \in \R} $ is the complete flow acting on $S\R^d \setminus \pi^{-1}(\mathring{D})$ which is defined as follows. For $(x,v) \in S\R^d\setminus \pi^{-1} (\mathring{D})$ set
\[
\tau_\pm(x,v) = \pm \inf\{t \geqslant 0: x \pm tv \in \partial D\}.
\]
For $(x, v) \in \mathcal D_\mathrm{in/out},$ denote by $v' \in \mathcal D_\mathrm{out/in}$  the image of $v$ by the reflexion with respect to $T_x(\partial D)$ at $x \in \partial D_j$,  given by
\[
v' = v - 2\langle v, n_j(x) \rangle n_j(x), \quad v \in S_x(\R^d), \quad x \in \partial D_j.
\]
Then,  for $(x,v) \in (S\R^d \setminus \pi^{-1}(D)) \cup \mathcal D_{\mathrm g},$ define
\[
\phi_t(x,v) = (x + tv, v), \quad t \in [\tau_-(x,v), \tau_+(x,v)],
\]
while for $(x,v) \in \mathcal D_{\mathrm{in/out}}$, we set
\[
\phi_t(x,v) = (x+tv, v) \quad \text{ if } \quad \left \lbrace \begin{matrix} (x, v) \in \mathcal D_\mathrm{out},~t \in \left[0, \tau_+(x,v)\right], \vspace{0.2cm} \\ \text{or } (x, v) \in \mathcal D_{\mathrm{in}},~t \in \left[\tau_{-}(x, v) , 0\right], \end{matrix} \right.
\]
and 
\[
\phi_t(x,v) = (x+tv', v') \quad \text{ if } \quad \left \lbrace \begin{matrix} (x, v) \in \mathcal D_\mathrm{in},~t \in \left]0, \tau_+(x,v)\right], \vspace{0.2cm} \\ \text{or } (x, v) \in \mathcal D_{\mathrm{out}},~t \in \left[\tau_{-}(x, v') , 0\right[. \end{matrix} \right.
\]
We extend $\phi_t$ to a complete flow still denoted by $\phi_t,$ having the property
\[
\phi_{t+s}(x,v) = \phi_t ( \phi_s(x,v)), \quad t,s \in \R, \quad (x,v) \in S\R^d \setminus \pi^{-1}(\mathring D).
\]

 Next, we introduce the non-grazing set $M$ as 
\[
M = B / \sim, \quad B = S\R^d \setminus \left(\pi^{-1}(\mathring{D}) \cup \mathcal D_\mathrm{g}\right),
\]
where $(x,v) \sim (y,w)$ if and only if $(x,v) = (y,w)$ or
\[
x = y \in \partial D \quad \text{ and } \quad w = v'.
\]
The set $M$ is endowed with the quotient topology. We change the notation and pass from $\phi_t$ to the non-grazing flow {$\varphi_t$}, which is defined on $M$ as follows. For $(x,v) \in (S\R^d \setminus \pi^{-1}(D)) \cup \mathcal D_\mathrm{in},$ define
\[
{\varphi_t([(x,v)])} = [\phi_t(x,v)], \quad t \in \left]\tau^{\mathrm{g}} _-(x,v), \tau^{\mathrm{g}} _+(x,v)\right[,
\]
where $[z]$ denotes the equivalence class of $z \in B$ for the relation $\sim$, and 
\[
\tau^{\mathrm{g}} _\pm(x,v) = \pm \inf\{t > 0:\phi_{\pm t}(x,v) \in \mathcal D_\mathrm{g}\}.
\]
Thus $\varphi_t$  is continuous, but the flow trajectory of  $(x, v)$  for times $t \notin \left]\tau^{\mathrm{g}} _-(x,v), \tau^{\mathrm{g}} _+(x,v)\right[$  is not defined. Clearly,  we may have  $\tau_{\pm}^{\mathrm{g}}(x, v)  = \pm\infty,$ while  $\tau^{\mathrm{g}}_{\pm}(x, v) \neq 0$ for $(x, v) \in {\mathcal D}_{\mathrm{in}}.$
Note that the above formula indeed defines a flow on $M$ because each $(x,v) \in B$ has a unique representative in $(S\R^d \setminus \pi^{-1}(\mathring{D})) \cup \mathcal D_\mathrm{in}$.Following \cite[Section 3]{delarue2022resonances}, we may define smooth charts on $M = B/\sim$ and $\varphi_t$ becomes $C^{\infty}$ non complete flow with respect to new charts.

Throughout, we work with the smooth flow $\varphi_t$ and denote by  $X$ its  the generator. Let $ A(z) = \{t \in \R~:~\pi(\varphi_t(z)) \in \partial D\}.$
 The {\it trapped set} $K$ of $\varphi_t$ is the set of points $z \in M$ which satisfy $-\tau_-^{\mathrm{g}}(z) = \tau_+^{\mathrm{g}} (z) = +\infty$ and
\[
\sup A(z) = - \inf A(z) = +\infty.
\]
By definition, $\varphi_t(z)$ is defined for all $t \in \R$ whenever $z \in K$. The flow $\varphi_t$ is called {\it uniformly hyperbolic} on $K$, if for each $z \in K$ there exists a decomposition
\begin{equation} \label{eq:2.1}
T_zM= \R X(z) \oplus E_u(z) \oplus E_s(z),
\end{equation}
which is $\dd \varphi_t$-invariant with $\dim E_s(z) = \dim E_u(z)  = d - 1$, {such that for some constants $C > 0,\: \nu > 0$, independent of $z \in K$}, and some smooth norm $\|\cdot\|$ on $TM$, we have
\begin{equation}\label{eq:2.6}
\left\|\dd \varphi_t(z) \cdot v\right\| \leqslant \left \lbrace \begin{matrix} C \e^{-\nu t} \|v\|,~ &v \in E_s(z),~&t\geqslant 0, \vspace{0.2cm}  \\  C \e^{-\nu |t|} \|v\|,~ &v \in E_u(z),~&t\leqslant 0. \end{matrix} \right.
\end{equation}
The spaces $E_s(z)$ and $E_u(z)$ depend continuously on $z$ (see \cite[Section 2]{hasselblat2002}).

The flow $\varphi_t$ is uniformly hyperbolic on $K$ (for the proof see \cite[Appendix A]{chaubet2022}). Take a small neighborhood $V$ of $K$ in $M$, with smooth boundary, and embed $V$ into a compact manifold without boundary $N$. We arbitrarily extend $X$ to obtain a smooth vector field on $N$, still denoted by $X$. The associated flow is still denoted by $\varphi_t$. Note that the new flow $\varphi_t$ is now complete. Introducing the surjective map
\[\pi_M: B \ni (x, \xi) \rightarrow [(x, \xi)] \in M,\] we have
$\varphi_t \circ \pi_M = \pi_M \circ \phi_t$ and there is a bijection between periodic orbits of $\phi_t$ and $\varphi_t$ preserving the periods of the closed trajectories of $\phi_t$, while the corresponding Poincaré maps are conjugated (see \cite[Section 3]{delarue2022resonances}).

 Consider the $(d-1)$-Grassmannian bundle 
\[\pi_G : G \to N\] over $N$. {More precisely}, for every $z \in N$, the set $\pi_G^{-1}(z)$ consists of all $(d-1)$-dimensional planes of $T_zN$.  The dimension of $\pi_G^{-1}(z) $ is $d(d-1)$
 and  $G$ is a smooth compact manifold with $\dim G = d^2 + d -1$ . We lift  $\varphi_t$ to a flow $\wt \varphi_t : G \to G$  defined by
\[
\wt \varphi_t(z,E) = (\varphi_t(z), \dd \varphi_t(z)(E)),\: z \in N,\: E \subset T_zN,\: \dd \varphi_t(z)(E) \subset T_{\varphi_t(z)}N.
\]
Introduce the set 
\[
\widetilde K_u  = \{(z,E_u(z))~:~z \in K\} \subset G.
\]
Clearly, $\wt K_u$ is invariant under the action of $\wt\varphi_t$, since $\dd\varphi_t(z) (E_u(z)) = E_u(\varphi_t(z))$. The set $\wt K_u$ will be seen as the trapped set of the restriction of $\widetilde \varphi_t$ to a neighborhood of $\wt K_u$ and the flow $\wt \varphi_t$ is uniformly hyperbolic on $\wt K$(see  \cite[Lemma A.3]{bowen1975ergodic}, \cite[\S2.5]{chaubet2022}). Let $\wt X$ be the generator of $\wt \varphi_t,$ and let  $\wt V_u$ be a small neighborhood of $\wt K_u$ in $G$ with smooth boundary $\partial \wt V_u$ (see\cite [\S2.7]{chaubet2022}). Define 
\[\Gamma_{\pm}(\wt X) = \{ z \in \wt V_u:\: \wt \varphi_t(z) \in \wt V_u,\: \mp t > 0\}.\]

 Denote by $\clos\: \wt V_u$ the closure of  $\wt V_u$. Let  $\tilde \rho \in C^\infty(\clos~\wt V_u, \bar{\R}_+)$ be the defining function for $\wt V_u$  such that  $\partial \wt V_u= \{z \in \clos~ \wt V_u: \tilde \rho(z) = 0\}$ and $\dd \tilde \rho(z) \neq 0$ for any $z \in \partial \wt V_u$. Following \cite[Lemma 2.3]{guillarmou2021boundary}, for any small neighborhood $\wt W_0$ of $\partial \wt V_u,$ there exists a vector field $\wt Y$ on $\clos \wt V_u$ arbitrary close to $\wt X$ in $C^{\infty}$-topology and flow $\wt \psi_t$ generated by $\wt Y$ with the properties:\\
\[ (1)\:\: {\rm supp}\:(\wt Y - \wt X) \subset \wt W_0. \]
(2)\:\: (Convexity condition)  For any defining  function $\rho$ of $\wt V_u$ and any $\omega \in \partial \wt V_u$ we have 
\[\wt Y \rho (\omega) = 0 \Longrightarrow \wt Y^2 \rho(\omega) < 0.\]
(3)\:\: $\Gamma_{\pm}(\wt X) = \Gamma_{\pm}(\wt Y),$ where $\Gamma_{\pm}(\wt Y)$ is defined as above by $\wt\psi_t.$\\

 By {\cite[Lemma 1.1]{dyatlov2016pollicott}}, we may find a smooth extension of $\wt Y$ on $G$ (still denoted by $\wt Y$) so that for every $\omega \in G$ and $T \geqslant 0$, we have
\begin{equation}\label{eq:globalconvex}
\omega, \wt\psi_T(\omega) \in \clos~ \wt V_u \quad \implies \quad \wt\psi_t(\omega) \in \clos~ \wt V_u ,\: \forall t \in [0, T].
\end{equation}
In the following, we fix $\wt V_u, \wt W_0, \tilde Y$ and the flow $\tilde \psi _t$ with the properties mentioned above. Thus, we obtain an {\it open hyperbolic system}
satisfying the conditions $(\bf A1) - (\bf A4)$ in \cite[\S0]{dyatlov2016pollicott} (see also \cite[\S2.1]{jin2023resonances}).

Next, repeating the setup in \cite[\S2.6]{chaubet2022}, we introduce some bundles passing to open hyperbolic system for bundles.
First, define the tautological vector bundle $\mathcal{E} \to G$ by
\[
\mathcal{E} = \{(\omega, u) \in \pi_{G}^*(TN)~:~\omega \in G,~u \in [\omega]\},
\]
where $[\omega] = E$ denotes the $(d-1)$ dimensional subspace of {$T_{\pi_G(\omega)}N$} represented by $\omega = (z,E)$ and $\pi_G^*(TN)$ is the pullback bundle of $TN.$
Second, introduce the ''vertical  bundle"  $\mathcal{F} \to G$ by
\[
\mathcal{F} = \{(\omega,W) \in TG~:~\dd \pi_G(\omega) \cdot W = 0 \},
\]
which is a subbundle of the bundle $TG \to G$. The dimensions of the fibres $\mathcal E_{\omega} $ and $\mathcal F_{\omega} $ of $\mathcal E$ and $\mathcal F$ over $\omega$ are given by
\[
\dim \mathcal E_{\omega} = d-1, \quad \dim \mathcal F_\omega = \dim \ker \dd \pi_G(\omega) = d^2 - d
\]
for any $\omega \in G$ with $\pi_G(\omega) = z$.
Finally, set 
\[
\mathcal{E}_{k, \ell} = \wedge^k \mathcal{E}^* \otimes \wedge^\ell \mathcal{F}, \quad 0 \leqslant k \leqslant d - 1, \quad 0 \leqslant \ell \leqslant d^2- d,
\]
where $\mathcal E^*$ is the dual bundle of $\mathcal E$, that is, we replace the fibre $\mathcal E_{\omega}$ by its dual space $\mathcal E_{\omega}^*$. 

Next, we use the notation $\omega = (z, \eta) \in G$ and $u \otimes v \in \mathcal E_{k, \ell}|_{\omega}$.
By using the flow $\tilde{\psi}_t$, introduce a flow
$\Phi^{k, \ell}_t : \mathcal E_{k, \ell} \to \mathcal E_{k, \ell}$
by 
\begin{equation}\label{eq:2.4}
\Phi_t^{k, \ell}(\omega, u \otimes v) = \Bigl(\tilde\psi_t(\omega), ~b_t(\omega) \cdot \left[ \left(\dd \varphi_{t}(\pi_G(\omega))^{-\top}\right)^{\wedge k}(u) \otimes \dd \wt \psi_{t}(\omega)^{\wedge \ell}(v)\right]\Bigr),
\end{equation}
with 
\[
b_t(\omega) = |\det \dd \varphi_t(\pi_G(\omega))|_{[\omega]}|^{1/2} \cdot |\det\left( \dd \tilde\psi_t(\omega)|_{\ker \dd \pi_G}\right)|^{-1},
\]
where $^{-\top}$ denotes the inverse transpose.
Consider the transfer operator
\[
\Phi^{k, \ell, *}_{-t} : C^\infty(G, \mathcal E_{k, \ell}) \to C^\infty(G, \mathcal E_{k, \ell})
\]
defined by
\begin{equation}\label{eq:2.5}
\Phi^{k, \ell, *}_{-t}\ubf(\omega) = \Phi^{k, \ell}_t \bigl[ {\bf u} (\wt \psi_{-t}(\omega)) \bigr], \quad {\bf u} \in C^\infty(G, \mathcal E_{k, \ell})
\end{equation}
and let $\mathbf{P}_{k, \ell} :  C^\infty(G, \mathcal E_{k, \ell}) \to C^\infty(G, \mathcal E_{k, \ell})$ be the generator of $\Phi^{k, \ell, *}_{-t}$ given by
\[
\Pbf_{k,\ell} {\bf u}  = \left.\frac{\dd}{\dd t} \Bigl(\Phi^{k, \ell,*}_{-t}{\bf u} \Bigr)\right|_{t=0}, \quad {\bf u} \in C^\infty(G, \mathcal E_{k, \ell}).
\]
The operator $\Pbf_{k, \ell}$ has the property 
\[\Pbf_{k,\ell} (f {\bf u}) = (\Pbf_{k, \ell} f) {\bf u} + f( \Pbf_{k, \ell} {\bf u}),\: f \in C^{\infty}(G),\: {\bf u} \in C^{\infty}(G, \mathcal E_{k, \ell}).\]

Thus, we obtain the same  setup as in Definition 6.1 in \cite[\S6.1]{jin2023resonances}. In the last paper the authors deal with a general Axiom A flow with several basic sets. In our case we have only one basic set and we may apply the results of \cite{dyatlov2016pollicott} and \cite{jin2023resonances}. With some constant $C > 0,$ we have
\[\|e^{- t \Pbf_{k,\ell}}\|_{L^2(G, \mathcal E_{k, \ell})
 \rightarrow L^2(G, \mathcal E_{k, \ell})} \leq C e^{C t},\: t \geq 0\]
 and
 \[(\Pbf_{k, \ell} + s)^{-1} = \int_0^{\infty} e^{-t(\Pbf_{k, \ell} + s)} dt : L^2(G, \mathcal E_{k, \ell})
 \rightarrow L^2(G, \mathcal E_{k, \ell}),\: \Re s \gg 1.\]
 Introduce the operator
\[
\Rbf_{k, \ell}(s) = {\bf 1}_{\wt V_u} (\Pbf_{k,\ell} + s)^{-1} {\bf 1}_{\wt V_u}: \: C^\infty_c(\wt V_u, \mathcal E_{k, \ell}) \rightarrow \mathcal{D}'(\wt V_u, \mathcal E_{k, \ell}),     \quad \Re(s) \gg 1,
\]
where ${\mathcal{D}}'(\wt V_u, \mathcal E_{k, \ell})$ denotes the space of $\mathcal E_{k, \ell}$-valued distributions. Applying \cite[Theorem 1]{dyatlov2016pollicott}, we obtain a meromorphic extension of  $\Rbf_{k, \ell}(s)$ to the whole plane $\C$ with simple poles and positive integer residues.

For $\omega \in G$ and $t > 0$ consider the {\it parallel transport} map
\[\alpha_{\omega, t}^{k, \ell} = \alpha_{1, \omega, t} \otimes \alpha_{2, \omega, t} :\: \Lambda^k \mathcal E_{\omega}^{*} \otimes \Lambda^{\ell}\mathcal F_ {\omega} \longrightarrow \Lambda^k \mathcal E_{{\tilde{\psi}}_t(\omega)} ^{*}\otimes \Lambda^{\ell}\mathcal F_{{\tilde{\psi}}_t(\omega)}\]
given by
\[ \ubf \otimes {\bf v}  \longmapsto (e^{- t \Pbf_{k, \ell}}(\ubf \otimes {\bf v} )) (\wt \psi (t)),\]
where $\ubf, {\bf v}$ are some sections of $\mathcal E_{\omega}^*$ and $\mathcal F_{\omega}$ over $\omega,$ respectively. The definition does not depend on the choice of $\ubf$ and ${\bf v}$ (see \cite[Eq. (0.8)]{dyatlov2016pollicott}). For a periodic trajectory $\wt\gamma: t \rightarrow \wt \gamma(t) = (\gamma(t), E_u(\gamma(t)))$ with period $T,$ we define 
\[\tr (\alpha^{k, \ell}_{\tilde{\gamma}})= \tr (\alpha_{\tilde{\gamma}(t), T}^{k, \ell})\] 
(see \cite{dyatlov2016pollicott}, \cite{chaubet2022}) and the trace is independent of the choice of the point $\wt \gamma(t) \in \wt \gamma.$

Finally, if $\wt \chi \in C_c^{\infty}(\wt V_u)$ is equal to 1 near the trapping set $\wt K_u,$
 we have the Guillemin trace formula (see \cite[(4.6)]{dyatlov2016pollicott},  \cite[\S3.1]{Weich2023},\cite [\S3.2]{chaubet2022}) with the flat trace
\begin{equation}\label{eq:2.6}
\tr^\flat (\wt\chi e^{-t \Pbf_{k, \ell}} \wt \chi)  = \sum_{\wt \gamma} \frac{\tau^\sharp(\gamma)\tr (\alpha^{k, \ell}_{\wt \gamma})\delta(t - \tau(\gamma))} {|\det(\id - \wt P_{\gamma})|}, \: t > 0.
\end{equation}
Here both sides are distributions on $(0, \infty)$ and the sum runs over all periodic orbits $\tilde{\gamma}$ of $\tilde{\varphi}_t$,
\[
\wt P_\gamma = \left.\dd \wt \varphi_{-\tau(\gamma)}(\omega_{\tilde{\gamma}})\right|_{\wt E_u(\omega_{\wt \gamma}) \oplus \wt E_s(\omega_{\tilde{\gamma}})}\]
is the linearized Poincar\'e map of the periodic orbit $\wt \gamma(t)$
of the flow $\wt \varphi_t$ and $\omega_{\tilde {\gamma}} \in \mathrm{Im}(\tilde \gamma)$ is any reference point taken in the image of $\tilde \gamma$. 

To treat the zeta function related only to periodic rays with number of  reflections $m(\gamma) \in q \N,\: q\geq 2,$ we consider the setup introduced in \cite[\S4.1]{chaubet2022} and we recall it below.  For $q \geqslant 2,$ define the \textit{$q$-reflection bundle} $\mathcal R_q \to M$ by
\begin{equation}\label{eq:2.7}
\mathcal R_q = \left. \left(\left[S\R^d \setminus \left(\pi^{-1}(\mathring{D}) \cup \mathcal D_\mathrm g \right)\right] \times \R^q\right) \right/ \approx,
\end{equation}
where the equivalence classes of the relation $\approx$ are defined as follows. For $(x,v) \in S\R^d \setminus \left(\pi^{-1}(\mathring{D}) \cup \mathcal D_\mathrm g \right)$ and $\xi \in \R^q$, we set
\[
[(x,v,\xi)] = \left\{(x,v, \xi) , (x, v', A(q) \cdot \xi)\right\} \quad \text{ if } (x,v) \in \mathcal D_\mathrm{in},\: (x, v') \in \mathcal D_\mathrm{out},
\]
where $A(q)$ is the $q \times q$ matrix with entries in $\{0, 1\}$ given by
\[A(q) = 
\begin{pmatrix} 
			    0 &   &  & 1\\
                              1 & 0 & & \\
                               & \ddots & \ddots & \\              
                               & & 1  & 0
                              \end{pmatrix}.
\]
{Clearly, the matrix $A(q)$ yields a shift permutation 
\[A(q) (\xi_1, \xi_2, ..., \xi_q) = (\xi_q, \xi_1,...,\xi_{q-1}).\]}
and
\begin{equation}\label{eq:2.8}
A(q)^q = \id, \quad \tr A(q)^j = 0, \quad j = 1, \dots, q-1.
\end{equation}
This indeed defines an equivalence relation since $(x,v') \in \mathcal D_\mathrm{out}$ whenever $(x,v) \in \mathcal D_\mathrm{in}$. Define a smooth structure of $\mathcal R_q$ as in \cite[\S4.1]{chaubet2022}
and introduce the bundle 
\[\mathcal E_{k, \ell}^q = \mathcal E_{k, \ell} \otimes \pi_G^*\mathcal R_q,\] where $\pi_G^*\mathcal R_q$ is the pullback of $\mathcal R_q$ by $\pi_G$ so $\pi_G^* \mathcal R_q \rightarrow G$ is a vector bundle over $G$.
Consider a small smooth neighborhood $V$ of $ K$.  We embed $V$ into a smooth compact manifold without boundary $N$, and we fix an extension of $\mathcal R_q$ to $N$. Consider any connection $\nabla^q$ on the extension of $\mathcal R_q$ which coincides with $\dd^q$ near $K$, and denote by 
\[
P_{q,t}(z) : \mathcal R_q(z) \to \mathcal R_q(\varphi_t(z))
\]
the parallel transport of $\nabla^q$ along the curve $\{\varphi_\tau(z)~:~0 \leqslant \tau \leqslant t\}$.
We have a smooth action of {$\varphi_t^q$ on $\mathcal R_q$ which is given by the horizontal lift of $\varphi_t$}
\[
\varphi_t^q(z, \xi) =(\varphi_t(z), P_{q,t}(z) \cdot \xi), \quad (z, \xi) \in \mathcal R_q.
\]

 We may lift the flow $\varphi_t$ to a flow $\Phi^{k, \ell, q}_t$ on $\mathcal E_{k, \ell}^q$ which is defined locally near $\wt K_u$ by
\[
\begin{aligned}
&\Phi^{k, \ell, q}_t(\omega, u \otimes v \otimes \xi) \\
 & \quad \quad = \Bigl(\wt \varphi_t(\omega), ~b_t(\omega) \cdot \left[ \left(\dd \varphi_{t}(\pi_G(\omega))^{-\top}\right)^{\wedge k}(u) \otimes (\dd \wt \varphi_t(\omega))^{\wedge \ell}(v) \otimes P_{q, t}(z) \cdot \xi \right]\Bigr)
 \end{aligned}
\]
for any $\omega = (z, E) \in G, \: u \otimes v \otimes \xi \in \mathcal E_{k, \ell}^q(\omega)$ and $t \in \R$.
Following \cite[\S4.1]{chaubet2022}, we deduce that for any periodic orbit $\gamma = (\varphi_\tau(z))_{\tau \in [0, \tau(\gamma)]}$, the trace 

\begin{equation}\label{eq:2.9}
{\tr (P_{q, \gamma})} = \tr (P_{q, \varphi(z)} )=
\left\{
\begin{matrix} q &\text{ if } & m(\gamma) = 0 \mod q, \vspace{0.2cm}\\ 
0 &\text{ if } & m(\gamma) \neq 0 \mod q
\end{matrix}
\right.
\end{equation}
 is independent of $z.$
Define the transfer operator
\[\Phi^{k, \ell, q, *}_{-t}: C^{\infty} (G, \mathcal E_{k, \ell}^q) \rightarrow C^{\infty} (G, \mathcal E_{k, \ell}^q)\]
by
\[\Phi^{k, \ell, q, *}_{-t}{\bf u} (\omega) = \Phi^{k, \ell, q}_{t}[{\bf u} (\tilde{\varphi}_{-t}(\omega)],\: {\bf u} \in C^{\infty}(G, \mathcal E_{k, \ell}^q)\]
and denote by $\Pbf_{k, \ell,q}$ be the generator of $\Phi^{k, \ell, q, *}.$ As above, we obtain the flat trace
\begin{equation}\label{eq:2.10}
\tr^\flat (\wt\chi e^{-t \Pbf_{k, \ell, q}} \wt \chi)  = q\sum_{\wt \gamma,  m(\pi_G(\wt \gamma)) \in q \N} \frac{\tau^\sharp(\gamma)\tr (\alpha^{k, \ell}_{\wt \gamma})\delta(t - \tau(\gamma))} {|\det(\id - \wt P_{\gamma})|}, \: t > 0.
\end{equation}

We close this section by the following 
\begin{lemm}[Lemma 3.1, \cite{chaubet2022}] \label{lem:alternatedsum}
 For any periodic orbit $\tilde{\gamma}$ of the flow $\tilde{\varphi}_t$ related to a periodic orbit $\gamma$,  we have
\[
\frac{1}{|\det(\id-\wt P_{\gamma})|}\sum_{k = 0}^{d-1}\sum_{\ell = 0}^{d^2 - d} (-1)^{k + \ell} \tr (\alpha_{\tilde{\gamma}}^{k, \ell}) = |\det(\id-P_\gamma)|^{-1/2}.
\]
\end{lemm} 

\section{Local trace formula}

In this section we apply the results of \cite{dyatlov2016pollicott} and {\cite[\S6.1]{jin2023resonances} for vector bundles. For simplicity, we will use the notations $\mathcal E_{k, \ell} = \mathcal E_{k, \ell}^1,\: \Pbf_{k, \ell} = \Pbf_{k, \ell, 1}$, etc.  For $\wt \chi \in C^\infty_c(\wt V_u)$ such that $\wt \chi \equiv 1$ near $\wt K_u$, by \cite{dyatlov2016pollicott} and \cite[\S6.1]{jin2023resonances},
 we conclude that for any integer $q \in \N$ 
\[
\wt \chi  (- i \Pbf_{k, \ell, q}  + s)^{-1} \wt \chi
\]
has a meromorphic continuation to $\C$. 
 Denote by $\Res(-i \Pbf_{k, \ell, q})$ the set of the poles of this continuation. Then, for any constant $\beta > 0,$ it was proved in \cite[(6.3)]{jin2023resonances} that we have the upper bound
\begin{equation} \label{eq:3.1}
\sharp \Res(- i \Pbf_{k, \ell, q} ) \cap \{\lambda \in \C,\: |\Re \lambda - E| \leq 1,\: \Im \lambda \geq - \beta \} = {\mathcal O}(E^{d^2+ d - 1}).
\end{equation}
In particular, there exists $C > 0$ depending on $\beta$ such that
\[ \sharp \Res(- i \Pbf_{k, \ell, q} ) \cap \{\lambda \in \C,\: | \lambda| \leq E,\: \Im \lambda \geq - \beta \} \leq C E^{d^2 + d}+ C.\]
Notice that the power $d^2 + d - 1$ comes from $\dim G$. Next, for $\Res(- i \Pbf_{k, \ell, q}),$ we obtain as in \cite{jin2023resonances} the following  local trace formula.
\begin{theo}[Theorem 1.5 and (6.5), \cite{jin2023resonances}] For every $ A > 0$ and any $q \in \N$ there exists a distribution $F_A^{k, \ell, q} \in {\mathcal S}'(\R)$ supported in $[0, \infty)$ such that
\begin{subequations}
\begin{align}
\sum_{\mu \in \Res(- i \Pbf_{k, \ell, q}), \Im \mu > -A} e^{- i\mu t} + F_A^{k, \ell, q}(t) \nonumber \\
= q\sum_{\tilde {\gamma}, \: m(\gamma) \in q\N} \frac{\tau^\sharp(\gamma)\tr (\alpha^{k, \ell}_{\tilde{\gamma}})\delta(t - \tau(\gamma))} {|\det(\id - \wt P_{\gamma})|}, \: t > 0.\tag{${3.2}_{k,\ell, q}$}
\end{align}
\end{subequations}
Moreover for any $\epsilon > 0$ the Fourier-Laplace transform $\hat{F}_A^{k, \ell, q}(\lambda)$ of $F_A^{k, \ell, q}(t)$ is holomorphic for $\Im \lambda < A - \epsilon$ and we have the estimate
\begin{equation}\label{eq:3.3}
|\hat{F}_A^{k, \ell, q}(\lambda) | = {\mathcal O}_{A, \epsilon, k, \ell, q} (1 + |\lambda|)^{2d^2 + 2d - 1 + \epsilon},\: \Im \lambda < A - \epsilon.
\end{equation}
Here $\gamma = \pi_G(\wt \gamma).$
\end{theo}
As  mentioned in  \cite[Section 6] {jin2023resonances}, the proof in \cite[Section 4] {jin2023resonances} with minor modifications works in the case of vector bundles.
Combining the above result with Lemma 2.1, we obtain the following theorem.
\begin{theo} For every $ A > 0$ and any $\epsilon > 0$ there exists a distribution $F_{A, q} \in {\mathcal S}'(\R)$ supported in $[0, \infty)$ with Fourier-Laplace transform $\hat{F}_{A, q}(\lambda)$ holomorphic for $\Im \lambda < A - \epsilon$ such that
\begin{subequations}\label{eq:3.4}
\begin{align}
\sum_{k = 0}^d \sum_{\ell = 0}^{d^2 - d} \sum_{\mu \in \Res(- i \Pbf_{k, \ell, q}), \Im \mu > -A} ( -1)^{k + \ell}e^{- i\mu t} + F_{A, q}(t)\nonumber \\
 = q\sum_{\gamma, \: m(\gamma) \in q\N} \frac{\tau^\sharp(\gamma)\delta(t - \tau(\gamma))} {|\det(\id -  P_{\gamma})|^{1/2}},\: t > 0,\tag{$(3.4)_q$}
 \end{align}
 \end{subequations}

where $\hat{F}_{A, q}(\lambda) = \sum_{k = 0}^d \sum_{\ell = 0}^{d^2 - d} (-1)^{k + \ell}\hat{F}_{A}^{k, \ell, q}(\lambda)$ satisfies the estimate $(\ref{eq:3.3})$.
\end{theo} 
Choosing $q = 1$, we obtain a local trace formula for Neumann dynamical zeta function $\eta_N(s)$, introduced in Section 1. For the Dirichlet dynamical zeta function $\eta_D(s)$ given in Section 1, we use the representation (\ref{eq:1.3})
and applying $(\ref{eq:3.4})_q$ with $q = 1, 2$, we obtain  the local trace formula
\begin{eqnarray} \label{eq:3.5}
\sum_{k = 0}^d \sum_{\ell = 0}^{d^2 - d} \sum_{\mu \in \Res(- i \Pbf_{k, \ell, 2}), \Im \mu > -A} ( -1)^{k + \ell}e^{- i\mu t}\nonumber \\
-\sum_{k = 0}^d \sum_{\ell = 0}^{d^2 - d} \sum_{\mu \in \Res(- i \Pbf_{k, \ell, 1}), \Im \mu > -A} ( -1)^{k + \ell}e^{- i\mu t} + F_{A, 2}(t) - F_{A, 1}(t)\nonumber \\
=\sum_{\gamma} \frac{(-1)^{m(\gamma)}\tau^\sharp(\gamma) \delta(t - \tau(\gamma))} {|\det(\id -  P_{\gamma})|^{1/2}}, \: t > 0.
\end{eqnarray}

Some resonances  $\mu \in \Res (- i \Pbf_{k, \ell, q}), \: k + \ell\: {\rm odd}, q = 1, 2$ may cancel with some resonances  $\nu \in\Res (-i \Pbf_{k, \ell, q}),\: k + \ell \: {\rm even}, q = 1, 2,$ and a priori it is not clear if the meromorphic continuation of dynamical zeta functions $\eta_N(s)$ and $\eta_D(s)$ have
 infinite number poles. Notice that all poles are simple and the cancellations in $(\ref{eq:3.4})_q$  could appear for terms with coefficients $+$ and $- $ related to $k + \ell$ even and $k + \ell$ odd, respectively. On the other hand, in (\ref{eq:3.5}) we have more possibilities for cancellations of poles.\\

 \section{Strip with infinite number poles}

\begin{proof}[Proof of Theorem 1.1] We will prove Theorem 1.1 for $\eta_D$ since the argument for $\eta_q$ is completely similar and simpler. 
After cancelation, all poles $ \mu$ at the left hand side of (\ref{eq:3.5}) satisfy $\Im \mu \leq \alpha = \max\{0, a_1\}.$  To avoid confusion, in the following we denote by $\wt \mu$ the poles $\mu$ in (\ref{eq:3.5}) which are not cancelled. Assume that for some $0 < \delta < 1$ and $0 \leq k \leq q, \; 0 \leq \ell \leq q^2 - q, \:  q = 1, 2,$ we have estimates
\begin{eqnarray} \label{eq:4.1}
N_{A, k, \ell, q} (r) = \sharp \{\wt \mu \in  \Res (-i \Pbf_{k, \ell, q}):\: |\wt \mu| \leq r, \:- A < \Im \wt \mu \leq \alpha\}\nonumber \\ 
\leq P(A, k, \ell, q, \delta) r^{\delta}.
\end{eqnarray}
We follow the argument in \cite[Section 5]{zworski2017local} and \cite[Appendix B]{chaubet2022} with some modifications. 
Let $\rho \in C_0^{\infty}(\R, \R_+)$ be an even function with ${\rm supp}\: \rho \subset \left[-1, 1\right]$ such that
\[
\rho(t) > 1 \quad \text{if} \quad |t| \leqslant 1/2,
\]
and 
\[
\hat{\rho}(-\lambda) = \int \e^{i t \lambda} \rho(t) \dd t \geqslant 0, \quad \lambda \in \R.
\]
Let $(\ell_j)_{j \in \N}$ and $(m_j)_{j \in \N}$ be sequences of positive numbers such that $\ell_j \geqslant d_0 = \min_{k \neq m} {\rm dist}\: (D_k, D_m)> 0, \: m_j \geqslant \max\{1, \frac{1}{d_0}\}$ and let $\ell_j \to \infty,\: m_j\to \infty$ as $j \to \infty$. Set
\[
\rho_j(t) = \rho(m_j (t- \ell_j)), \quad t \in \R,
\]
and introduce the distribution $\mathcal F_\mathrm D \in  {\mathcal S}'(\R^+)$ by
\begin{equation} \label{eq:4.1}
\mathcal F_\mathrm D(t) = \sum_{\gamma \in \mathcal P} \frac{(-1)^{m(\gamma)} \tau^{\sharp}(\gamma)  \delta(t - \tau(\gamma))}{|\det(I - P_{\gamma})|^{1/2}}.
\end{equation}
The following proposition has been established by Ikawa.

\begin{prop}[Prop. 2.3, \cite{ikawa1990poles}]\label{prop:2.3ika} Suppose that the function $s \mapsto \eta_\mathrm D(s)$ cannot be prolonged as an entire function of $s$. Then there exists $\alpha_0 > 0$ such that for any $\beta > \alpha_0$ we can find sequences  $(\ell_j), (m_j)$ with $\ell_j \to \infty$ as $j \to \infty$  such that for all $j \geqslant 0$ one has
\[
e^{\beta \ell_j} \leqslant m_j \leqslant \e^{2 \beta \ell_j}
\quad \text{and} \quad
|\langle \mathcal F_\mathrm D, \rho_j \rangle | \geqslant \e^{- \alpha_0 \ell_j}.
\]
\end{prop} 
We apply the local trace formula (\ref{eq:3.5}) to function $\rho_j(t).$ For $- A \leq \Im \zeta \leq \alpha$ we have
\[|\hat{\rho}_j(\zeta) | = m_j^{-1}  |\hat{\rho}(m_j^{-1}  \zeta) e^{- i \ell_j \zeta} | \leq C_N m_j^{-1} e^{\alpha\ell_j+m_j^{-1}\max (\alpha, A)}( 1 + |m_j^{-1} \zeta|)^{-N}.\]
Then, for $q = 1, 2$ and $- A \leq \Im \wt \mu \leq \alpha,$  one obtains
\[\big |\sum_{\Im \wt \mu > -A, \:\wt \mu \in \Res(- i \Pbf_{k, \ell, q})} \langle e^{- i \wt \mu t}, \rho_j(t)\rangle\big | \]
\[\leqslant C_{N, A} m_j^{-1} e^{ \alpha \ell_j}\int_0^{\infty} (1 + m_j^{-1}  r)^{-N} d N_{A, k, \ell, q}(r) \]
\[=- C_{N, A} m_j^{-1} e^{\alpha \ell_j} \int_{0}^{\infty} \frac{d}{dr} \Bigl( (1 + m_j^{-1}  r)^{-N}\Bigr) N_{A, k, l, q}(r)dr \]
\[\leqslant B_{N, A} P(A, k,\ell, q, \delta) m_j^{-(1 -\delta)}e^{\alpha\ell_j }\int_0^{\infty} ( 1 + y)^{-N- 1}y^{\delta} dy \]
\[= A_N P(A, k, \ell, q, \delta) m_j^{-(1- \delta)}e^{\alpha \ell_j}\leq D_{A, k, \ell, q, \delta} e^{(-\beta(1-\delta) + \alpha)\ell_j}.\]

Next, applying (\ref{eq:3.3}), we have
\[\langle F_{A, q}, \rho_j \rangle = \int_{\R} \hat{F}_{A, q}(-\zeta) \hat{\rho}_j(\zeta) d \zeta = \int_{\R + i(\epsilon - A)} \hat{F}_{A, q}(-\zeta) \hat{\rho}_j(\zeta) d\zeta\]
and choosing $M = 2d^2 + 2d + 1$, we deduce
\[|\langle F_{A, q}, \rho_j \rangle | \leq C_{M, A,  q} m_j^{-1} e^{ (\epsilon - A)\ell_j} e^{m_j^{-1} \max\{A - \ep, \alpha\}} \]
\[ \times \int (1 + |\zeta|)^{2d^2 + 2d - 1 +\epsilon} ( 1 + |m_j^{-1} \zeta|)^{-M} d\zeta\]
\[\leq D_{M, A, q} e^{ (\epsilon - A) \ell_j}m_j^{2d^2 + 2d -1 + \epsilon}\leq D_{M, A, q}  e^{ (\epsilon - A)\ell_j }e^{2(2d^2 + 2d -1 + \epsilon)\beta \ell_j}.\]
If the function $\eta_D$ cannot be prolonged as entire function, we may apply Proposition 4.1. Taking together the above estimates and summing for $0 \leq k \leq d,\: 0 \leq \ell \leq d^2- d$ and $q = 1, 2$, we get
\[D_{A} e^{(-\beta(1-\delta) + \alpha)\ell_j} + E_A e^{(\epsilon - A)\ell_j }e^{2(2d^2 + 2d -1 + \epsilon)\beta \ell_j} \geq e^{-\alpha_0 \ell_j}.\]
Here, the constants $D_A$ and $E_A$ depend on $A$ but they are independent of $\ell_j.$  Choose $\beta > \frac{\alpha_0 + \alpha}{1 - \delta}$ and fix $\beta$ and $0 < \epsilon < 1$. Next choose
\[A > 2(2d^2 + 2d -1 + \epsilon)\beta +  \epsilon + \alpha_0.\]
Fixing $A$, for $\ell_j \to \infty$ we obtain a contradiction. This completes the proof of Theorem 1.1 for $\eta_D.$\\

To deal with $\Res \eta_q, \: q \geq 2$, we choose a periodic ray $\gamma_0$ with $q$ reflections, $\ell_j = j \tau^{\sharp}(\gamma_0), \: m_j = e^{ \beta \ell_j}$  The existence of a periodic ray with $q$ reflexions follows from the fact that for every configuration $\xi \in \Sigma_A$  there exists unique ray $\gamma(\xi)$ with successive reflexion points on $....\partial D_{j-1}, \partial D_j, \partial D_{j + 1},...$ (see the Appendix for the definition of $\Sigma_A$ and \cite{ikawa1988fourier}). We apply the lower bound
\[\big |\langle \sum_{\gamma, \: m(\gamma) \in q\N} \frac{\tau^\sharp(\gamma)\delta(t - \tau(\gamma))} {|\det(\id -  P_{\gamma})|^{1/2}}, \rho_j\rangle \big | \geq c e^{-c_0 \ell_j}, \: \forall j \geq 0\]
with $c > 0, c_0 > 0$ independent of $\ell_j.$ For $ q = 1,$ we choose $\ell_j = j \tau^{\sharp}(\gamma),\: m_j = e^{\beta \ell_j}$ with some periodic ray $\gamma$ and obtain the above estimate. Repeating the argument for $\eta_D$, we prove (\ref{eq:1.4}). \end{proof}

\begin{proof}[Proof of Theorem $1.2$] We follow the approach of F. Naud in \cite[Appendix A]{zworski2017local}. Let $0 \leq \rho \in C_0^{\infty}(-1, 1)$ be the function introduced above. For $\xi \in \R$ and $t > \max\{d_0, 1\}$ introduce the function 
\[\psi_{t, \xi}(s) = e^{ i\xi s} \rho(s- t), \: \xi \in \R.\]
We apply the trace formula $(\ref{eq:3.4})_1$  to $\psi_{t, \xi}$. As above, denote by $\wt\mu$ the poles which are not cancelled in the left hand side of $(\ref{eq:3.4})_1$. Assume that for $0 \leq k \leq d$ and $0 \leq \ell \leq d^2 - d,$ we have
\begin{equation}
\sharp \{ \wt \mu \in \Res(-i \Pbf_{k, l}): \: - A- \ep \leq \Im \wt \mu \leq \alpha\} = P(A, k, \ell, \ep) < \infty.
\end{equation} 
 First, we have
 \[|\hat{\psi}_{t, \xi}(\zeta) | \leq C_N e^{t \Im \zeta + |\Im \zeta|} (1 + |\Re \zeta - \xi|)^{-N}.\]
For $- A \leq \Im \wt\mu \leq \alpha$ and $N = 1$ the sum of terms involving poles $\wt\mu$ in $(\ref{eq:3.4})_1$ can be bounded by $\frac{C_1e^{\alpha t}}{1+ |\xi |}$ with constant $C_1 > 0$ depending on $P(A, k, \ell, \ep)$ and $\exp(\max\{A, \alpha\})$. Second, by using (\ref{eq:3.3}) for $\hat{F}_{A, 1}$, one deduces
\[|\langle F_{A,1}, \psi_{t, \xi}\rangle| \leq C_2 e^{(\ep - A)( t- 1)}(1 + |\xi|)^{2d^2 + 2d - 1 + \ep}.\]
Setting
\[S(t, \xi) = \sum_{\gamma} \frac{e^{i \xi \tau(\gamma)} \tau^{\sharp}(\gamma)\rho(\tau(\gamma) - t)}{|\det(\id -  P_{\gamma})|^{1/2}},\]
we get
\[|S(t, \xi)| \leq \frac{C_1 e^{\alpha t}}{1 + |\xi|} + C_A e^{-(A - \ep)t}(1 + |\xi|)^{2d^2 + 2d - 1 + \ep}.\]
Now consider the Gaussian weight
\[G(t, \sigma) = \sigma^{1/2} \int_{\R} |S(t, \xi)|^2 e^{-\sigma\xi^2/2} d\xi,\: 0 < \sigma < 1.\]
The estimate for $|S(t, \xi)|$ yields
\[|S(t, \xi)|^2 \leq \frac{2 C_1^2e^{2\alpha t}}{(1 + |\xi|)^2} + 2 C_A^2e^{-2(A -\ep)t}(1 + |\xi|)^{2(2d^2 + 2d - 1 + \ep)}\]
and
\begin{equation}\label{eq:4.4}
G(t, \sigma) \leq C_1'\sigma^{1/2}e^{2\alpha t} + C_A'\sigma^{-(2 d^2 + 2d - 1 + \ep)} e^{-2(A - \ep)t}.
\end{equation}
On the other hand, taking into account only the terms with $\tau(\gamma) = \tau(\gamma')$, we get
\begin{equation}
\begin{aligned} \label{eq:4.5}
G(t, \sigma) = \sqrt{2 \pi} \sum_{\gamma} \sum_{\gamma'} \frac{\tau^{\sharp}(\gamma) \tau^{\sharp}(\gamma')e^{-(\tau(\gamma) - \tau(\gamma'))^2/2\sigma} \rho(\tau(\gamma) - t) \rho(\tau(\gamma')- t) }{ |\det(\id -  P_{\gamma})|^{1/2}|\det(\id -  P_{\gamma'})|^{1/2}} \\
\geq c \sum_{t-1/2 \leq \tau(\gamma) \leq t +1/2} \tau^{\sharp}(\gamma)|\det(\id -  P_{\gamma})|^{-1}
\end{aligned}
\end{equation}
with $ c > 0$ independent of $t$ and $\sigma.$

Set $\tau(\gamma) = T_{\gamma},\: \tau^{\sharp}(\gamma) = T^{\sharp}_\gamma,\:a_{\gamma} = \frac{T^{\sharp}_\gamma}{|\det(\id -  P_{\gamma})|}.$
Recall that $b_1$ is the abscissa of convergence of Dirichlet series (\ref{eq:1.5}) with $q = 1.$\\
\def\32{\frac{3}{2}}

We need the following 
\begin{lemm} Let $0 < \ep$ be sufficiently small and $b_1 \neq 0.$ Then there exists a sequence $t_j \to \infty$ such that
\begin{equation} \label{eq:4.6}
\sum_{t_j - 1/2\leq \tg  \leq t_j + 1/2} \ag  \geq e^{(b_1- 2\ep)t_j}.
\end{equation}
On the other hand, for $b_1 = 0$ and small $u > 0$ there exists a sequence $t_j \to \infty$ such that
we have the estimate $(\ref{eq:4.6})$ with $b_1$ replaced by $-u/2.$

\end{lemm} 

\begin{proof} We consider three cases.\\
{\bf Case 1.} $b_1 < 0.$\\
Let the lengths of the periodic rays be arranged as follows
\[T_{1} \leq T_{2} \leq ...\leq T_{n} \leq ....\]
It is well known (see for instance, \cite{cotton1917}) that
\[b_1 = \limsup_{n \to \infty}  \frac{ \log |\sum_{T_n \leq \tg} \ag |}{T_n}.\]
We fix a small $\epsilon > 0$ so that $- \delta = b_1 - 3\epsilon/2 < 0$.There exists an increasing  sequence $n_1 < n_2< ...< n_m <...$ such that $\lim n_j = + \infty$ and
\begin{equation} \label{eq:4.7}
\frac{ \log |\sum_{T_{n_j} \leq \tg} a_{\gamma} |}{T_{n_j}} \geq b_1 - \epsilon.
\end{equation}
 Choose $n_1$ large so that 
\[1 > e^{-\delta} + 2e^{-\frac{\ep}{2} T_{n_j}},\:  e^{\frac{\ep}{2} T_{n_j}}\geq e^{\delta/2}\: {\rm for}\:  j \geq m.\]
Set $q_1 = T_{n_1}$ and write
\[\sum_{q_1 \leq \tg} \ag = \sum_{k= 0}^{\infty} \sum_{q_1 +k \leq \tg <  q_1+ k + 1} \ag.\]
Assume that we have the estimates
\begin{equation}\label{eq:4.8}
\sum_{ q_1 + k \leq \tg < q_1 + k + 1} \ag \leq e^{-\delta (q_1 + k)},\: \forall k \geq 0.
\end{equation}
Then,
\[ \sum_{q_1 \leq \tg } \ag   \leq e^{-\delta q_1} \sum_{k = 0}^{\infty} e^{- k \delta}= e^{-\delta q_1} \frac{1}{1 - e^{-\delta}} < \frac{1}{2}e^{(-\delta + \epsilon/2) q_1} .\]
Since $- \delta + \epsilon/2 = b_1 - \epsilon,$ we obtain a contradiction with (\ref{eq:4.7}) for $T_{n_1}$. Consequently, there exists at least one $k_1 \geq 0$ such that
\begin{equation} \label{eq:4.9}
\sum_{ q_1 + k_1\leq \tg <  q_1 + k_1 + 1} \ag > e^{- \delta (q_1 + k_1)}.
\end{equation}

The series $\sum_{\tg \geq (k_1+1) q_1} \ag e^{- \lambda \tg}$ has the same abscissa of convergence $b_1$. We repeat the procedure choosing $q_2 >  q_1 + k_1 + 2$, and obtain the existence of $k_2 \geq 0$ such that
\begin{equation*} 
\sum_{ q_2 + k_2\leq \tg <  q_2+ k_2 + 1} \ag > e^{-\delta (q_2+ k_2)}.
\end{equation*}
 By iteration, we find two sequences $\{q_j\},\: \{k_j\}$ such that \[ q_{j+1} >  q_j + k_j + 2, \]
  and a sequence of disjoint intervals
 \[[q_j + k_j,  q_j + k_j + 1], \: j = 1, 2,...\]
so that 
\begin{equation} \label{eq:4.10}
\sum_{q_j + k_j\leq \tg \leq  q_j+ k_j + 1} \ag > e^{- \delta( q_j + k_j)}.
\end{equation}
The periods $q_j$ may change in the above procedure but for simplicity we use the same notation.
Choosing $t_j = q_j+ k_j + 1/2,$ we deduce (\ref{eq:4.6})

{\bf Case 2.} $b_1 > 0.$\\
For $b_1$  we have the representation
\[b_1 = \limsup_{n \to \infty} \frac{\log |\sum_{\tg \leq T_n} a_{\gamma} |}{T_n}.\]
We fix a small $\ep > 0$ so that $d_1 = b_1 - \frac{3}{2} \ep > 0.$
There exists an increasing  sequence $n_1 < n_2< ...< n_m <...$ such that $\lim n_j = + \infty,\: \frac{e^{d_1}}{e^{d_1} - 1} < e^{\frac{\ep}{2}[T_{n_1}]}$ and
\begin{equation} \label{eq:4.11}
\frac{ \log |\sum_{\tg \leq T_{n_j}} a_{\gamma} |}{T_{n_j}} \geq b_1 - \epsilon.
\end{equation}
We get
\[\sum_{\tg \leq T_{n_1}} a_{\gamma} \leq \sum_{k= 0}^{[T_{n_1}]} \sum_{k < \tg \leq k + 1} a_{\gamma}.
\]
 Assume that for $ k = 0,...,[T_{n_1}]$ we have
\[\sum_{k < \tg \leq k+1} a_{\gamma} \leq e^{d_1 k}.\]
This implies
\[\sum_{\tg \leq T_{n_1}} a_{\gamma} \leq \sum_{k = 0}^{[T_{n_1}]} e^{d_1 k}  = \frac{e^{([T_{n_1}] + 1)d_1} - 1}{e^{d_1} - 1} 
< \frac{e^{d_1} e^{-\frac{\ep}{2} [T_{n_1}]}}{e^{d_1} - 1} e^{(b_1 - \ep)T_{n_1}}.\]
and we obtain a contradiction with (\ref{eq:4.11}) and $T_{n_1}.$ Thus for some 
$ 0 \leq k_1 \leq [T_{n_1}] $ we have
\[\sum_{k_1 < \tg \leq k_1+1} a_{\gamma} > e^{(b_1 - \frac{3}{2} \ep)k_1}.\]
Following this procedure, we construct a sequence of integers $\{k_j\},\: k_{j+ 1} \geq T_{k_j}+ 1$ satisfying
\[\sum_{k_j \leq \tg \leq k_j + 1} a_{\gamma} > e^{(b_1- \frac{3}{2} \ep)k_j}.\]
and choosing  $t_j = k_j + 1/2$, we  arrange (\ref{eq:4.6}) for large $k_j.$ 

{\bf Case 3.} $b_1 = 0.$\\
For small $u > 0,$ consider the Dirichlet series
\[\eta_{u}(s) = \sum_{\gamma} \frac{\tg^{\sharp} e^{-(s + u)\tg}}{|\det(\mathrm{\id} - P_{\gamma})|} = \sum_{\gamma} a_{\gamma} e^{- u \tg}e^{-s\tg}.\]
This series has abscissa de convergence $ - u < 0$ and we may apply the results of  Case 1. For a suitable sequence $t_j \to \infty$ depending on $-u,$ we obtain the estimates
\[e^{- u(t_j - 1/2)} \sum_{t_j - 1/2\leq \tg  \leq t_j + 1/2} \ag  \geq \sum_{t_j - 1/2\leq \tg  \leq t_j + 1/2} \ag e^{- u \tg} \geq e^{(-u- 2\ep)t_j}.\]
Consequently,
\[\sum_{t_j - 1/2\leq \tg  \leq t_j + 1/2} \ag  \geq e^{-u/2}e^{-2\ep t_j} > e^{(-u/2 - 2 \ep)t_j}.\]
 \end{proof}

Going back to the proof of Theorem 1.2, assume first that $b_1 \neq 0.$ Therefore from (\ref{eq:4.4}) and (\ref{eq:4.5}) with $t = t_j$ we obtain
\begin{equation}\label{eq:4.15}
c_1 \sigma^{1/2} e^{2\alpha t_j}+ c_2 \sigma^{-(2d^2 + 2 d - 1 + \ep)} e^{-2(A - \ep)t_j} \geq e^{(b_1 - 2\ep)t_j}
\end{equation}
with constants $c_1, c_2 > 0$ independent of $t_j.$
Now choose
 \[\sigma = c_1^{-2} e^{2(b_1- 3\ep- 2\alpha)t_j} < 1.\] 
 Since
\[(b_1 - 3 \ep - 2\alpha) -(b_1 - 2\ep) + 2\ep  \leq  \ep,\]
we have
\[e^{-\ep t_j} + c_3 e^{-2(2d^2 + 2d -1/2+ \ep)(b_1 - 3 \ep - 2\alpha)t_j} e^{-2(A - (1/2)\ep) t_j }\geq 1.\]
Taking
\[A = -(2d^2 + 2d - 1/2)(b_1 - 2 \alpha)+ 3\ep(2d^2 + 2d - \frac{b_1- 2 \alpha}{3}+  \ep)\]
and letting $t_j \to +\infty$, we obtain a contradiction. Consequently, for some $0 \leq k_0 \leq 0,\: 0 \leq \ell_0 \leq d^2 - d$, setting $\tilde{\ep} = 3\ep(2d^2 + 2d - \frac{b_1 - 2 \alpha}{3} + \ep) + \ep$, we have
\[
\begin{aligned}
\sharp \{ \wt \mu \in \Res(-i \Pbf_{k_0, l_0}): \: \Im \wt \mu > (2d^2 + 2d- 1/2) (b_1- 2\alpha) - \tilde{\ep}\} = \infty.
\end{aligned}
\]
This implies (\ref{eq:1.6}) with $\ep$ replaced by $\tilde\ep$, observing that the poles $\wt \mu \in \Res(-i \Pbf_{k_0,\ell_0})$ coincide with the poles $\wt \lambda$ of the meromorphic continuation of
$\eta_1(-i \lambda)$. \\

For $b_1 = 0$ the estimates (\ref{eq:4.6}) hold with $b_1$ replaced by $-u/2$. The argument in the Case 1 implies
\[\sharp\{\mu_j \in \Res\: \eta_1:\:\Re \mu_j > (2d^2 +2 d -1/2)( -2\alpha)- (\ep + (2d^2 + 2d - 1/2) u/2)\} = \infty.\]
For small $u,$ we arrange $(2d^2 + 2d -1/2) u/2  <\ep,$ and since $\ep$ is arbitrary, we obtain (\ref{eq:1.6}) with $b_1 = 0.$ This completes the proof of Theorem 1.2. \end{proof}

\renewcommand{\theprop}{A.\arabic{prop}}
\renewcommand{\therem}{A.\arabic{rem}}  
\renewcommand{\theequation}{\arabic{section}.\arabic{equation}}
\section*{ Appendix }  \renewcommand{\theequation}{A.\arabic{equation}}
\setcounter{equation}{0}

  Here, we prove Proposition 1.2. 
  \begin{proof}[Proof of Proposition $1.2$]
  First,
 \[\det(\id - P_{\gamma}) = \det(\id - D_x \varphi_{T_{\gamma}}\vert_{E_s(x)}) \det(\id - D_x \varphi_{T_{\gamma}}\vert_{E_u(x)})\]
 \[= \det (D_x\varphi_{T_{\gamma}}\vert_{E_u(x)} )\det(\id - D_x \varphi_{T_{\gamma}}\vert_{E_s(x)})  \det(D_x \varphi_{-T_{\gamma}}\vert_{E_u(x)}- \id), \: x \in \gamma.\] 
 Consequently,
 \[ |\det(\id - P_{\gamma})|^{-1} = |\det D_x\varphi_{T_{\gamma}}\vert_{E_u(x)} |^{-1} \]
 \[\times|\det (\id -D_x\varphi_{T_{\gamma}}\vert_{E_s(x)}) |^{-1} |\det(\id-D_x \varphi_{-T_{\gamma}}\vert_{E_u(x)}) |^{-1}.\]
 For large $\tg$ we have
 \[\|D_x \varphi_{\tg}\vert_{E_s(x)}\| \leq C e^{-\delta \tg},\:\|D_x \varphi_{-\tg}\vert_{E_u(x)}\| \leq C e^{-\delta \tg},\: \delta > 0,\forall \tg\]
 with constants $ C >0, \: \delta > 0$ independent of $\tg$ since the flow $\varphi_t$ is uniformly hyperbolic (see \cite[Appendix A]{chaubet2022}). Thus, for large $\tg$ we obtain
 \begin{equation} \label{eq:A.1}
 c_1|\det D_x\varphi_{T_{\gamma}}\vert_{E_u(x)}|^{-1} \leq  |\det(\id - P_{\gamma})|^{-1} \leq C_1|\det D_x\varphi_{T_{\gamma}}\vert_{E_u(x)}|^{-1}
 \end{equation} 
 with $0 < c_1 < C_1$ independent of $\tg.$
 We have
 \[\det D_x\varphi_{T_{\gamma}}\vert_{E_u(x)} = e^{d_{\gamma}},\: x \in \gamma\]
 with 
 \[d_{\gamma} = \log \Bigl(\lambda_{1, \gamma}...\lambda_{d-1, \gamma}) > 0,\]
 where $\lambda_{j, \gamma}$ are the eigenvalues of $D_x \varphi_{\tg}\vert_{E_u(x)}$ with modulus greater than 1. The above estimate shows that the abscissa of convergence of the series
 \[
 \sum_{\gamma} \tg^{\sharp} e^{-s \tg + \delta_{\gamma}},\: \delta_{\gamma} = - d_{\gamma}, \: \Re s \gg 1
 \]
 coincides with $b_1$.

 Our purpose is to express $b_1$ by some dynamical characteristics related to symbolic dynamics for several disjoint strictly convex obstacles. To do this, we recall some well known results and we refer to \cite{ikawa1988zeta}, \cite{ikawa1988fourier}, \cite{ikawa1990poles}, \cite{Parry1990} for more details. Let $A(i, j)_{i, j = 1,...,r}$
be an $r \times r$ matrix such that
\[ A(i, j) = \begin{cases}  1 \:\: {\rm if}\: i \neq j,\\
0 \:\: {\rm if}\: i = j.\end{cases} \]
Introduce the spaces
\[\Sigma_A = \{ \xi = \{\xi_i\}_{i = -\infty}^{\infty},\:\:\xi_i \in \{1,,...,r\},\:\: A(\xi_i, \xi_{i +1}) = 1,\: \forall i \in \Z\},\]
 \[\Sigma_A^{+} = \{ \xi = \{\xi_i\}_{i = 0}^{\infty},\:\:\xi_i \in \{1,,...,r\},\:\: A(\xi_i, \xi_{j + 1}) = 1,\: \forall i \geq 0\}.\] 
 Given $0 < \theta < 1$, define a metric $d_{\theta}$ on $\Sigma_A$ by $d_{\theta}(\xi, \eta) = 0$ if $\xi = \eta$ and $d_{\theta}(\xi, \eta) = \theta^k$ if $\xi \neq \eta,$ and let $k$ be the maximal integer such that $\xi_i= \eta_i$ for $| i | < k.$ Similarly, we define a metric $d_{\theta}^{+}$ on $\Sigma_A^+.$ Following   \cite[Chapter 1]{Parry1990},  for a function $F: \Sigma_A \rightarrow \C,$ define
 \[{\rm var}_k F = \sup\{|F(\xi) - F(\eta)|: \: \xi_i = \eta_i,\: | i| < k\}\]
 and for $G: \Sigma_A^{+} \rightarrow \C,$ define
 \[{\rm var}_k G = \sup\{|G(\xi) - G(\eta)|: \: \xi_i = \eta_i,\: 0 \leq i < k\}.\]
 Let $F_{\theta}(\Sigma_A),\: F_{\theta}(\Sigma_A^+)$ be the set of Lipschitz functions with respect to metrics $d_{\theta}, d_{\theta}^+,$ respectively, with norm
 \[\|| f\||_{\theta} = \|f\|_{\infty} + \|f\|_{\theta},\: \|f\|_{\theta} = \sup_{k \geq 0} \frac{{\rm var}_k f}{\theta^k}.\]
  Let $\sigma_A$ be a shift on $\Sigma_A$ and $\Sigma_A^{+}$  given by
 \[(\sigma_A \xi)_i =\xi_{i+1},\:\forall i \in \Z,\: (\sigma_A \xi)_i = \xi_{i+1}, \: \forall i \geq 0 ,
 \]
 respectively. For every $\xi \in \Sigma_A$ there exists a unique reflecting ray $\gamma(\xi)$ with successive reflections points on $....\partial D_{j-1}, \partial D_j, \partial D_{j + 1},...$, where the order of reflections is determined by the sequence $(\xi)$  (see \cite{ikawa1988fourier}). If $(P_j(\xi))_{j = -\infty}^{\infty}$ are the reflexion points of $\gamma(\xi)$, we define the function
 \[f(\xi) = \|P_0(\xi) - P_1(\xi)\|.\] 
 It was proved in \cite[Section 3]{ikawa1988fourier} and \cite[Section 3]{petkov2010resolvent} that one can construct a sequence of phase functions
 $\{\varphi_{\xi, j}(x)\}_{j = -\infty}^{\infty},$
 such that for each $j,$ the phase $\varphi_{\xi, j}$ is smooth in a neighborhood ${\mathcal U}_{\xi, j}$ of the segment $[P_j(\xi), P_{j+1}(\xi)]$ in $\R^d \setminus \mathring D$ and\\
 
 (i) $\|\nabla \varphi_{\xi, j}(x)\| = 1$ on ${\mathcal U}_{\xi, j}$,\\
 
 (ii) $\nabla\varphi_{\xi, j}(P_j(\xi))= \frac{P_{j+1}(\xi)- P_j(\xi)}{\|P_{j+1}(\xi)- P_j(\xi)\|},$\\
 
 (iii) $\varphi_{\xi, j} = \varphi_{\xi, j+1} $ on $\partial D_{j+1} \cap {\mathcal U}_{\xi, j} \cap {\mathcal U}_{\xi, j+1},$\\
 
 (iv) for each $x \in {\mathcal U}_{\xi, j}$ the surface $C_{\xi, j}(x) = \{y \in {\mathcal U}_{\xi, j}:\: \varphi_{\xi, j}(x) = \varphi_{\xi, j}(y)\}$ is strictly convex with respect to its normal fields $\nabla \varphi_{\xi, j}.$

 Denote by $\kappa_j(\xi) ,\: j = 1,...,d-1,$ the principal curvatures at $P_0(\xi)$ of $C_{\xi, 0}(x)$ 
 and introduce
 \[g(\xi) = - \log \prod_{j = 1}^{d-1} (1 + f(\xi)\kappa_j(\xi)).\] 
 Then,
 \[\prod_{j= 1}^{d-1} \lambda_{j, \gamma(\xi)} = \prod_{k = 1}^{m(\gamma(\xi))}\prod_{j = 1}^{d-1}( 1 + f(\sigma_A^k\xi)\kappa_j(\sigma_A^k\xi)).\]

 It follows form the exponential instability of the billiard ball map (see \cite{ikawa1988fourier}, \cite{stoyanov1999instability}) that $f(\xi), g(\xi)$ become functions in $F_{\theta}(\Sigma_A)$ with $0 < \theta < 1$ depending on the geometry of $D$. 
 We define
 \[S_n h(\xi) = h(\xi) + h(\sigma_A \xi) +...+h(\sigma_A^{n-1}\xi)\]
 and for a periodic ray $\gamma(\xi),$ we obtain
 \[T_{\gamma(\xi)} = S_{m(\gamma(\xi))} f(\xi),\:\delta_{\gamma(\xi)} = S_{m(\gamma(\xi))} g(\xi).\]
 Consider the zeta function
 \[Z(s) = \Bigl(\sum_{n=1}^{\infty} \frac{1}{n} \sum_{\sigma_A^n\xi = \xi} e^{S_n(-s f(\xi) + g(\xi))}\Bigr), \: \Re s \gg 1\]
 and observe that 
 \[- \frac{d}{ds} Z(s) = \sum_{\gamma} \tg^{\sharp}e^{-s\tg + \delta_{\gamma}}.\]
 Next, it is well known (see for instance \cite[Chapter 1]{Parry1990}) that given $h \in F_{\theta}(\Sigma_A),$ there exist functions $\tilde{h}, \chi \in F_{\theta^{1/2}}(\Sigma_A)$ such that
 \[h(\xi) = \tilde{h}(\xi) + \chi(\sigma_A \xi) - \chi(\xi)\]
 and $\tilde{h}(\xi) \in F_{\theta^{1/2}}(\Sigma_A^+)$ depends only on the coordinates $(\xi_0,\xi_1,...).$ We denote this property by $h \sim \tilde{h}.$ Choose $\tilde{f} \sim f,\: \tilde{g} \sim g$ with $\tilde{f}, \tilde{g} \in F_{\theta^{1/2}}(\Sigma_A^+)$ and write 
 \[Z(s) = \Bigl( \sum_{n=1}^{\infty} \frac{1}{n} \sum_{(\sigma_A^+)^n\xi = \xi} e^{S_n(-s \tilde{f}(\xi) + \tilde{g}(\xi))}\Bigr).\] 
  The pressure $\P(F)$ of a function $F \in C(\Sigma_A)$ is defined by
 \[\P(F) = \sup_{\nu} \Bigl(h(\nu, \sigma_A) + \int_{\Sigma_A} F d \nu\Bigr),\]
 where $h(\nu, \sigma_A)$ is the measure entropy of $\sigma_A$ with respect to $\nu$ and the supremum is taken over all probability measures $\nu$ on $\Sigma_A$ invariant with respect to $\sigma_A$.
 
 Following \cite[Chapter 6]{Parry1990}, consider the suspended flow $\sigma^f_t (\xi, s) = (\xi, s + t)$ 
  on the space
 \[\Sigma_A^f = \{(\xi, s):\: \xi \in \Sigma_A,\: 0 \leq s \leq f(\xi)\}\]
 with identification $(\xi, f(\xi)) \sim(\sigma_A(\xi), 0).$ For a function $G \in C(\Sigma_A^f),$ define the pressure
 \begin{equation} \label{eq:A.2}
 \P(G) = \sup_{\nu_f}\{ h(\nu_f, \sigma^f_t) + \int_{\Sigma_A^f} G d \nu_f\},
 \end{equation}
 where $h(\nu_f, \sigma^f_t)$ is the measure entropy and the supremum is taken over all probability measures $\nu_f$ on $\Sigma_A^f$ invariant with respect to $\sigma^f_t.$  
  The suspended flow $\sigma^f_t$ is weakly mixing, if there are no $t \in \R \setminus \{0\}$ with the property
 \[\frac{t}{2\pi} f(\xi) \sim M(\xi),\]
 where $M(\xi)\in C(\Sigma_A: \Z)$ has only integer values.  According to \cite[Lemma 5.2]{stoyanov1999instability} and \cite[Lemma 1]{petkov1999zeta}, the flow $\sigma^f_t$ is {\it weakly mixing}.
 
 Applying the results of \cite[Chapter 6]{Parry1990}, we deduce that the abscissa of convergence $b_1$ of $Z(s)$ is determined as the root of the
 equation $\P(- s \tilde{f} + \tilde{g}) = \P(- s f + g) = 0$ with respect to $s$. This root is unique since $s \rightarrow \P(- sf + g)$ is decreasing.
 Introduce the function
 \begin{equation} \label{eq:A.3}
 G(\xi, y) = -\frac{1}{2}\sum_{j = 1}^{d-1} \frac{\kappa_j(\xi)}{1 + \kappa_j(\xi) y}.
 \end{equation}
 Clearly,
 \[g(\xi) = 2\int_0^{f(\xi)} G(\xi, y) dy.\]
 Then \cite[Proposition 6.1]{Parry1990} says that $\P(-b_1 f + g) = 0$ is equivalent to $b_1 = \P(2G).$ With the same argument we show that $a_1 = \P(G).$ This completes the proof of Proposition 1.2. \end{proof}
 
It easy to find a relation between $\P(2G)$ and $\P(g)$. Repeating the argument of \cite[Section 3]{petkov2010resolvent}, one obtains that there exist probability measures $\nu_g,\:\nu_0$ on $\Sigma _A$ invariant with respect to $\sigma_A$ such that
\[ \frac{\P(g)}{\int f(\xi) d \nu_g} \leq b_1 \leq \frac{\P(g)}{\int f(\xi) d\nu_0}.\]
 Consequently, $b_1$ has the same sign as $\P(g).$
 
We close this Appendix by proving that $b_1 = b_2.$ Consider the zeta function
 \[Z_1(s) =  \sum_{n=1}^{\infty} \frac{1}{n} \sum_{(\sigma_A^+)^n\xi = \xi} e^{S_n(-s \tilde{f}(\xi) + \tilde{g}(\xi) + i \pi)}\]  
 related to  (\ref{eq:1.9}). Introduce the complex Ruelle operator
 \[(\mathcal L_s u) (\xi) = \sum_{\sigma_A^+ \eta = \xi} e^{ (- s \tilde{f} + \tilde{g} + i \pi)(\eta)}u(\eta), \: u \in F_{\theta^{1/2}}(\Sigma_A^+).\]
 Then for $s = b_1,$ this operator has no eigenvalues 1 since this implies that the operator
 \[( L_{b_1} u) (\xi) = \sum_{(\sigma_A^+) \eta = \xi} e^{ (- b_1 \tilde{f} + \tilde{g})(\eta)} u(\eta)\]
 will have eigenvalue (-1). This is impossible because from $\P(-b_1 \tilde{f} + \tilde{g}) = 0,$ one deduces that $L_{b_1} $ has eigenvalue 1 and all other eigenvalues of $L_{b_1}$ have modulus strictly less than 1 (see \cite[Theorem 2.2]{Parry1990}).This shows that the function $Z_1(s)$ is analytic for $s =b_1$, hence (\ref{eq:1.9}) has the same property. Finally, similarly to (\ref{eq:1.3}), we write the function (\ref{eq:1.9}) as a difference of two Dirichlet series with abscissas of convergence $b_1$ and $b_2$. Therefore the inequality $b_2 < b_1$ leads to contradiction.
 
  \subsection*{Acknowledgements.} We would like to thank Long Jin, Frédéric Naud and Zhongkai Tao  for the useful and stimulating discussions. Thanks are due to referee for his/her valuable remarks and suggestions.
  
\bibliographystyle{alpha}
\bibliography{bib.bib}

\begin{thebibliography}{AFW17}

\bibitem[AFW17]{arnoldi2017}
Jean~Francois Arnoldi, Fr\'{e}d\'{e}ric Faure, and Tobias Weich.
\newblock Asymptotic spectral gap and {W}eyl law for {R}uelle resonances of
  open partially expanding maps.
\newblock {\em Ergodic Theory Dynam. Systems}, 37(1):1--58, 2017.

\bibitem[Ber33]{bernstein1933}
Vladimir Bernstein.
\newblock {\em Le\c cons sur les progrès récents de la théorie de séries de
  Dirichlet, professées au Collège de France}.
\newblock Gauthier-Villars, 1933.

\bibitem[BR75]{bowen1975ergodic}
Rufus Bowen and David Ruelle.
\newblock The ergodic theory of {A}xiom {A} flows.
\newblock {\em Invent. Math.}, 29(3):181--202, 1975.

\bibitem[Cot17]{cotton1917}
Emile Cotton.
\newblock Sur l'abscisse de convergence des séries de {D}irichlet.
\newblock {\em Bulletin Soc. Math.}, 45:121--125, 1917.

\bibitem[CP22]{chaubet2022}
Yann Chaubet and Vesselin Petkov.
\newblock Dynamical zeta functions for billiards.
\newblock {\em https://doi.org/10.48550/arXiv.2201.00683}, 2022.
\newblock Ann. Inst. Fourier, (Grenoble), to appear.

\bibitem[DG16]{dyatlov2016pollicott}
Semyon Dyatlov and Colin Guillarmou.
\newblock {Pollicott--Ruelle resonances for open systems}.
\newblock {\em Annales Henri Poincar{\'e}}, 17(11):3089--3146, 2016.

\bibitem[DSW24]{delarue2022resonances}
Benjamin Delarue, Philipp Schütte, and Tobias Weich.
\newblock Resonances and weighted zeta functions for obstacle scattering via
  smooth models.
\newblock {\em Annales Henri Poincaré}, 25(2):1607--1656, 2024.

\bibitem[FT13]{faure2013band}
Fr\'{e}d\'{e}ric Faure and Masato Tsujii.
\newblock Band structure of the {R}uelle spectrum of contact {A}nosov flows.
\newblock {\em C. R. Math. Acad. Sci. Paris}, 351(9-10):385--391, 2013.

\bibitem[FT17]{faure2017semiclassical}
Fr{\'e}d{\'e}ric Faure and Masato Tsujii.
\newblock The semiclassical zeta function for geodesic flows on negatively
  curved manifolds.
\newblock {\em Inventiones mathematicae}, 208(3):851--998, 2017.

\bibitem[GMT21]{guillarmou2021boundary}
Colin Guillarmou, Marco Mazzucchelli, and Leo Tzou.
\newblock Boundary and lens rigidity for non-convex manifolds.
\newblock {\em American Journal of Mathematics}, 143(2):533--575, 2021.

\bibitem[Has02]{hasselblat2002}
Boris Hasselblatt.
\newblock {\em Hyperbolic dynamical systems}.
\newblock North-Holland, Amsterdam, 2002.

\bibitem[Ika88a]{ikawa1988fourier}
Mitsuru Ikawa.
\newblock Decay of solutions of the wave equation in the exterior of several
  convex bodies.
\newblock {\em Ann. Inst. Fourier (Grenoble)}, 38(2):113--146, 1988.

\bibitem[Ika88b]{ikawa1988zeta}
Mitsuru Ikawa.
\newblock On the existence of poles of the scattering matrix for several convex
  bodies.
\newblock {\em Proc. Japan Acad. Ser. A Math. Sci.}, 64(4):91--93, 1988.

\bibitem[Ika90a]{ikawa1990poles}
Mitsuru Ikawa.
\newblock On the distribution of poles of the scattering matrix for several
  convex bodies.
\newblock In {\em Functional-analytic methods for partial differential
  equations ({T}okyo, 1989)}, volume 1450 of {\em Lecture Notes in Math.},
  pages 210--225. Springer, Berlin, 1990.

\bibitem[Ika90b]{ikawa1990zeta}
Mitsuru Ikawa.
\newblock Singular perturbation of symbolic flows and poles of the zeta
  functions.
\newblock {\em Osaka J. Math.}, 27(2):281--300, 1990.

\bibitem[Ika00]{ikawa2000resonances}
Mitsuru Ikawa.
\newblock On scattering by several convex bodies.
\newblock {\em J. Korean Math. Soc.}, 37(6):991--1005, 2000.

\bibitem[JT25]{jin2023resonances}
Long Jin and Zhongkai Tao.
\newblock Counting {P}ollicott-{R}uelle resonances for {A}xiom {A} flows.
\newblock {\em Comm. Math. Physics}, 406, (2), 2025.
\newblock article 26.

\bibitem[JZ17]{zworski2017local}
Long Jin and Maciej Zworski.
\newblock A local trace formula for {A}nosov flows.
\newblock {\em Annales Henri Poincar\'{e}}, 18(1):1--35, 2017.
\newblock With appendices by Fr\'{e}d\'{e}ric Naud.

\bibitem[Mor91]{morita1991symbolic}
Takehiko Morita.
\newblock The symbolic representation of billiards without boundary condition.
\newblock {\em Transactions of the American Mathematical Society},
  325(2):819--828, 1991.

\bibitem[Pet99]{petkov1999zeta}
Vesselin Petkov.
\newblock Analytic singularities of the dynamical zeta function.
\newblock {\em Nonlinearity}, 12(6):1663--1681, 1999.

\bibitem[Pet02]{petkov2002poles}
Vesselin Petkov.
\newblock Lower bounds on the number of scattering poles for several strictly
  convex obstacles.
\newblock {\em Asymptot. Anal.}, 30(1):81--91, 2002.

\bibitem[PP90]{Parry1990}
William Parry and Mark Pollicott.
\newblock Zeta functions and the periodic orbit structure of hyperbolic
  dynamics.
\newblock {\em Ast\'{e}risque}, (187-188), 1990.
\newblock 268 pp.

\bibitem[PS10]{petkov2010resolvent}
Vesselin Petkov and Luchezar Stoyanov.
\newblock Analytic continuation of the resolvent of the {L}aplacian and the
  dynamical zeta function.
\newblock {\em Anal. PDE}, 3(4):427--489, 2010.

\bibitem[PS17]{petkov2017geometry}
Vesselin~M. Petkov and Luchezar~N. Stoyanov.
\newblock {\em Geometry of the generalized geodesic flow and inverse spectral
  problems}.
\newblock John Wiley \& Sons, Ltd., Chichester, second edition, 2017.

\bibitem[Sto99]{stoyanov1999instability}
Lachezar Stoyanov.
\newblock Exponential instability for a class of dispersing billiards.
\newblock {\em Ergodic Theory Dynam. Systems}, 19(1):201--226, 1999.

\bibitem[Sto01]{stoyanov2001spectrum}
Luchezar Stoyanov.
\newblock Spectrum of the {R}uelle operator and exponential decay of
  correlations for open billiard flows.
\newblock {\em American Journal of Mathematics}, 123(4):715--759, 2001.

\bibitem[Sto09]{stoyanov2009poles}
Luchezar Stoyanov.
\newblock Scattering resonances for several small convex bodies and the
  {L}ax-{P}hillips conjecture.
\newblock {\em Memoirs Amer. Math. Soc.}, 199(933), 2009.
\newblock vi+76 pp.

\bibitem[Sto12]{stoyanov2012non}
Luchezar Stoyanov.
\newblock Non-integrability of open billiard flows and {Dolgopyat}-type
  estimates.
\newblock {\em Ergodic Theory and Dynamical Systems}, 32(1):295--313, 2012.

\bibitem[SWB23]{Weich2023}
Philipp Sch\"{u}tte, Tobias Weich, and Sonja Barkhofen.
\newblock Meromorphic continuation of weighted zeta functions on open
  hyperbolic systems.
\newblock {\em Comm. Math. Phys.}, 398(2):655--678, 2023.

\end{thebibliography}

\end{document}